\documentclass[12pt, a4paper]{amsart}
\usepackage[latin1]{inputenc}

\usepackage[T1]{fontenc}
\usepackage{amsthm}
\usepackage{amsfonts}
\usepackage{graphicx,color}
\usepackage{amssymb}
\usepackage{amscd}
\usepackage{amsmath}
\usepackage{color}
\usepackage{algorithm}
\usepackage{algpseudocode}
\usepackage{enumerate}
\usepackage{hyperref}
\usepackage{dsfont}
\usepackage{caption}
\usepackage{subcaption}
\usepackage{enumitem}
\usepackage{cite}
\usepackage{tikz}

\usetikzlibrary{automata}

\allowdisplaybreaks

\newtheorem{theorem}{Theorem}[section]
\newtheorem*{theorem*}{Theorem}
\newtheorem{lemma}[theorem]{Lemma}
\newtheorem{corollary}[theorem]{Corollary}
\newtheorem{definition}[theorem]{Definition}

\theoremstyle{definition}
\newtheorem{example}[theorem]{Example}
\newtheorem{remark}[theorem]{Remark}

\newcommand{\itemref}[1]{\ref{#1}}

\newcommand{\Z}{\mathbb{Z}}
\newcommand{\R}{\mathbb{R}}
\newcommand{\C}{\mathbb{C}}
\newcommand{\bigOh}{\mathcal{O}}
\newcommand{\eps}{\varepsilon}
\newcommand{\Expect}{\mathbb{E}}
\newcommand{\Var}{\mathbb{V}}
\DeclareMathOperator{\Cov}{Cov}
\DeclareMathOperator{\grad}{grad}
\DeclareMathOperator{\sign}{sign}
\DeclareMathOperator{\diag}{diag}
\newcommand{\T}{\mathcal{T}}
\newcommand{\inputsum}{\mathsf{Input}}
\newcommand{\out}{\mathsf{Output}}
\newcommand{\mmatrix}{I-\frac zK\sum_{\eps\in\mathcal A_{I}}x^{\eps}M_{\eps}(y)}
\newcommand{\bfomega}{\boldsymbol{\Omega}_{n}}
\newcommand{\bfs}{\boldsymbol{s}}
\newcommand{\bfzero}{\boldsymbol{0}}
\newcommand{\bfone}{\boldsymbol{1}}

\newcommand{\mw}{\mathsf{mw}}
\renewcommand{\MR}[1]{}

\title[Variances and Covariances of the Output of a Transducer]{Variances and Covariances in the Central
  Limit Theorem for the Output of a Transducer}
\author{Clemens Heuberger}
\address{Institut f\"ur Mathematik, Alpen-Adria-Universit\"at Klagenfurt, Austria}
\email{clemens.heuberger@aau.at}
\thanks{The first two authors are supported by the Austrian Science Fund (FWF):
  P~24644-N26.}
\author{Sara Kropf}
\address{Institut f\"ur Mathematik, Alpen-Adria-Universit\"at Klagenfurt, Austria}
\email{sara.kropf@aau.at}
\author{Stephan Wagner}
\address{Department of Mathematical Sciences, Stellenbosch University, South
  Africa}
\email{swagner@sun.ac.za}
\thanks{The third author is supported by the National Research Foundation of South Africa under grant number 70560.}
\keywords{Hamming weight, variance, covariance, independence, Matrix-Tree Theorem, transducer,
central limit theorem, Quasi-Power Theorem.}
\subjclass[2010]{60F05; 68W40 05C20 05C30 11B85 68Q45}

\begin{document}

\begin{abstract}
We study the joint distribution of the input sum and the output sum of a
deterministic transducer. Here, the input of this finite-state machine
is a uniformly distributed random sequence.

We give a simple combinatorial characterization of transducers for which the output sum has bounded variance, and we
also provide algebraic and combinatorial characterizations of transducers for which the covariance of input and output sum
is bounded, so that the two are asymptotically independent.

Our results are illustrated by several examples, such as transducers that count specific blocks in the binary expansion, the transducer that computes
the Gray code, or the transducer that computes the Hamming weight of the width-$w$ non-adjacent form digit expansion. 
The latter two turn out to be examples of asymptotic independence.
\end{abstract}
  
\maketitle

  \section{Introduction}
We consider sequences defined as the sum of the output of a deterministic transducer, i.e. a finite-state machine that deterministically transforms an input sequence into an output
sequence. Here, we let both the input and the output be sequences of real numbers and assume that the input sequence is randomly generated.
Then, while the output depends deterministically on the input, the dependence between the two random variables ``sum of the input sequence'' and ``sum of the output sequence'' may become negligible for long input sequences. We investigate for which transducers this is the case. We give two different
characterizations of such ``independent'' transducers, an algebraic and a
combinatorial one. In a similar way, we also consider the
variance of the sum of the output of a transducer. We prove a combinatorial
characterization of transducers with bounded variance of the output
sum. These combinatorial characterizations are described in terms of a
weighted number of functional digraphs or cycles of the underlying graph.

Our probability model is the equidistribution on all input sequences of a fixed
length $n$. We asymptotically investigate the two random variables ``sum of the input'' and ``sum of
the output'' of a transducer for $n\to\infty$. If these two random variables
converge in distribution to
independent random variables, then the transducer is called independent.

Under this probability model, the expected value of the sum
of the input and the output are $e_{1}n$ and $e_{2}n+\bigOh(1)$,
respectively, for some constants $e_{1}$ and $e_{2}$. For the sum of the
input, the expressions are exact without error term because the input letters
are independent and identically distributed. Furthermore, under appropriate
connectivity conditions, the variances and the covariance turn out to be  $v_{1}n$,
$v_{2}n+\bigOh(1)$ and  $cn+\bigOh(1)$, respectively, for suitable constants
$v_{1}$, $v_{2}$ and $c$. We investigate for which transducers one of the
constants $v_{2}$ and $c$ is zero.

A special case
of the output sum is the Hamming weight, which is the number of nonzero elements of a sequence. 
To give an example of an independent transducer, we later discuss the Hamming weight of
the non-adjacent
form as defined by Reitwiesner~\cite{Reitwiesner:1960}. The non-adjacent form is
the unique digit expansion with digits $\{-1,0,1\}$, base $2$ and the
syntactical rule that at least one of any two adjacent digits has to be
$0$. It has minimal Hamming weight among all digit expansions with digits
$\{-1,0,1\}$ in base $2$. In~\cite{Heuberger-Prodinger:2007:hammin-weigh}, Heuberger and Prodinger
prove that the Hamming weights of the standard binary expansion and the
non-adjacent form are asymptotically independent.  The independent transducer computing
these Hamming weights is shown in Figure~\ref{fig:NAF}.

There are many results on the variance of the sum of the output of explicit
transducers under the same probability model we use. See,
for example, \cite{Grabner-Heuberger-Prodinger:2004:distr-results-pairs,Grabner-Heuberger-Prodinger-Thuswaldner:2005:analy-linear,Avanzi-Heuberger-Prodinger:2006:scalar-multip-koblit-curves,Heigl-Heuberger:2012:analy-digit} for the variance of the Hamming weight
of different digit expansions which are computed by
transducers. In~\cite{Heuberger-Prodinger:2006:analy-alter}, the authors count
the occurrences of a digit and give the expected value, the variance and the
covariance between two different digits. The occurrence of a specific
pattern in a word is investigated in e.g.~\cite{Bender-Kochman:1993:distr-subwor,Nicodeme-Salvy-Flajolet:2002:motif,Flajolet-Szpankowski-Vallee:2006:hidden-word-statis,Goldwurm-Radicioni:2007:averag-value}
(with generalizations to other probability
models, too). In~\cite{Bender-Kochman:1993:distr-subwor}, the covariance
between different patterns is also considered. 
In~\cite{Grabner-Thuswaldner:2000:sum-of-digits-negative}, Grabner and Thuswaldner
consider a transducer whose output is the sum of digit function. However, they were only interested in the 
output and did not consider the joint distribution or the covariance of the
input and output sum.

By contrast, we are interested
in the joint distribution of the input and output sum for a general transducer. We not only algebraically
compute the expected value and the variance-covariance matrix of this
distribution, but we also give combinatorial descriptions of these
values. In particular, we combinatorially characterize independent transducers
and transducers with bounded variance of the output sum.
This combinatorial connection is described by a condition on some weighted number of
functional digraphs or on each cycle of the underlying graph of the
transducer. 
To obtain these results, we apply a generalization of the
Matrix-Tree Theorem by Chaiken~\cite{Chaiken:1982:matrixtree} and
Moon~\cite{Moon:1994:matrixtree}.

We formally define our setting in the next section. In
Section~\ref{sec:main-results}, we state our main results. In
Section~\ref{sec:ex}, we present several examples where these main
results are applied. In the last section, we give the proofs of the theorems.

In many contexts, an unbounded
 variance (as in~\cite{Hwang:1998}) is necessary to prove a Gaussian limit law. In Theorem~\ref{thm:var}, we combinatorially describe transducers whose output sums have bounded variance. For strongly connected transducers, we prove that this is the case if and only if
there exists a constant such that for each cycle, the output sum is
proportional to its length with this proportionality constant. This in turn is
equivalent to a quasi-deterministic output sum in the sense that the difference of
the output sum and its expected value is bounded for \emph{all} events, independently of
the length of the input.
In the special case where the transducer is strongly connected and aperiodic and the only possible outputs are $0$ and $1$, it turns out that the output sum has asymptotically bounded variance if and only
if the output is constant for all transitions (Corollary~\ref{cor:01output}). 
The assumption of strong connectivity can be relaxed for most results.

We give an
algebraic description of independent transducers in Theorem~\ref{thm:alg}. We
also state there that the input sum and the output sum are asymptotically jointly normally distributed
if the variance-covariance matrix is invertible. In Theorem~\ref{thm:comb}, we present a combinatorial characterization of
independent transducers.

In Section~\ref{sec:ex}, we give a variety of examples of
independent and dependent transducers and transducers with bounded and
unbounded variance to illustrate our results. One of those examples is a transducer
computing the minimal Hamming weight of $\tau$-adic digit representations on
a digit set $\mathcal D$. Building on the results of~\cite{Heigl-Heuberger:2012:analy-digit}, we prove that the variance of the minimal
Hamming weight is unbounded, which yields a central limit theorem.

In Section~\ref{sec:proofs-theorems}, we also prove an extension of the
$2$-dimensional Quasi-Power Theorem~\cite{Heuberger:2007:quasi-power} to
singular Hessian matrices as an auxilliary result.

The results of Theorem~\ref{thm:alg} have been implemented \cite{Heuberger-Kropf:trac-16145} in the computer
algebra system \textsf{Sage}~\cite{Stein-others:2014:sage-mathem-6.3}, based on its
package for finite state
machines~\cite{Heuberger-Krenn-Kropf:ta:finit-state}. This code is included in
\textsf{Sage}~6.3.

\section{Preliminaries}\label{sec:preliminaries}

A \emph{transducer} is defined to consist of a finite set of states
$\{1,2,\ldots,S\}$, a finite input alphabet $\mathcal A_{I}\subseteq\R$, an output alphabet
$\mathcal A_{O}\subseteq \R$,
 a set of transitions $\mathcal E\subseteq\{1,2,\ldots,S\}^{2}\times\mathcal A_{I}$ with input labels in $\mathcal A_{I}$, output labels
$\delta\colon \mathcal E\to \mathcal A_{O}$ and the initial state $1$. The
transducer is called \emph{deterministic} if for all states $s$ and input labels
$\eps\in\mathcal A_{I}$, there exists at most one state $t$ such that
$(s,t,\eps)\in\mathcal E$.
Furthermore, the transducer is said to be \emph{subsequential} (cf.~\cite{Schuetzenberger:1977}) if it is
deterministic, every state is final and it has a final output $a\colon \{1,2,\ldots,S\}\to
\mathcal A_{O}$. A transducer is called \emph{complete} if for every state $s$ and
digit $\eps\in\mathcal A_{I}$, there is a transition from $s$ to a state $t$ with
input label $\eps$, i.e.,
$(s,t,\eps)\in\mathcal E$.

\begin{definition}
A transducer is said to be \emph{finally connected} if there exists a state
which can be reached from any other state. The \emph{final component} of such a
transducer is defined to be the transducer induced by the set of
states which can be reached from any other state. A finally connected
transducer is said to be \emph{finally aperiodic} if the underlying graph of the final
component is aperiodic (i.e., the gcd of the lengths of all walks starting and ending at a given vertex is $1$).
\end{definition}

\begin{remark}
The
final component of a transducer is a strongly connected component of the underlying graph of the
transducer.
If the underlying graph is strongly connected, then being finally aperiodic is
equivalent to being aperiodic. We then call the transducer strongly connected
and aperiodic. The final component of a complete transducer is
complete itself.
\end{remark}

In the following, we consider subsequential, complete, deterministic, finally connected, finally
aperiodic transducers. We require that the input
alphabet $\mathcal A_{I}$ has at least
two elements. Throughout the paper, we use $\eps$ for the input of a transition and
$\delta$ for the output of a transition. We denote the number of states in the final
component by $N$.

The input of the transducer is a sequence in $\mathcal A_{I}^{*}$. It is not
important whether we read the input from right to left or in the other direction,
we just have to fix it for one specific transducer. The output of the
transducer is the sequence of output labels of the unique path starting at the
initial state $1$ with the given input as input label, together with the final
output label of the final state of this path.

\begin{definition}
The \emph{Hamming weight} of a finite sequence $(n_{0},\ldots,n_{L})$ is the number of nonzero elements of the sequence.
\end{definition}

\begin{example}\label{ex:first}
The transducer in Figure~\ref{fig:NAF} is a subsequential, complete, strongly
connected, aperiodic transducer. It
computes the Hamming weight of the non-adjacent form when reading the binary
expansion from right to left. The transducer is a slight simplification of
the one in, e.g.,~\cite{Heuberger-Prodinger:2007:hammin-weigh}, taking into
account that we are only interested in the Hamming weight.

For example, the non-adjacent form of $12$ is $(10(-1)00)_{\text{NAF}}$ and
has Hamming weight $2$. When reading the standard binary
expansion $(1100)_{2}$ of $12$, the transducer in
Figure~\ref{fig:NAF} writes the output $10100$. The leftmost $1$ in the output
is the final output of the last state. The sum of the output is $2$, too.
\end{example}

Let $X_{n}$ be a uniformly distributed random variable on
$\mathcal A_{I}^{n}$. Let $\out(X_{n})$ be the sum of the output sequence of the
transducer if the input is $X_{n}$. Furthermore, let $\inputsum(X_{n})$ be the sum
of the input sequence. Without loss of generality, we fix the direction of
reading from right to left.

We investigate the $2$-dimensional random vector
\begin{equation*}\bfomega=(\inputsum(X_{n}),\out(X_{n}))^{t}\end{equation*} for $n\to\infty$, where ${}^{t}$ denotes
transposition. We will prove that each component of this random vector either converges in
distribution to a normally distributed random variable or to a degenerate
random variable. Here, a random variable is said to be \emph{degenerate} if it is
constant with probability $1$. By definition, a degenerate random variable is independent of
any other random variable. Thus, the variance of a degenerate random variable
and the covariance of a degenerate and any other random variable are always
$0$.

For a finally connected, aperiodic transducer, the expected value and the variance of
$\bfomega$ will turn out to be
$(e_{1}, e_{2})^{t}n+\bigOh(1)$, $(v_{1}, v_{2})^{t}n+\bigOh(1)$,
respectively, for suitable constants $e_{1}$, $e_{2}$, $v_{1}$ and
$v_{2}$ (see Theorem~\ref{thm:alg}). The covariance between the two coordinates will be $cn+\bigOh(1)$ for some constant $c$.
We call $\Sigma=\bigl(\begin{smallmatrix}v_{1}&c\\c&v_{2}\end{smallmatrix}\bigr)$ the \emph{asymptotic variance-covariance matrix} of $\bfomega =(\inputsum(X_n), \out(X_n))^t$. Its
entries are called the \emph{asymptotic variances} and the \emph{asymptotic
  covariance}.

A special case is a transducer with output alphabet $\{0,1\}$. 
If we consider
a transducer computing, for example, a new digit expansion and we are only
interested in the Hamming weight of this new digit expansion, we map the
output of each transition to the alphabet $\{0,1\}$. In this way, we obtain a
new transducer with output alphabet $\{0,1\}$ computing the Hamming weight of
this new digit expansion. In this case, the combinatorial characterization of
a bounded variance of the output sum is particularly simple (see Corollary~\ref{cor:01output}).

For brevity, we introduce the notion of independent transducers.
\begin{definition}
A transducer is \emph{independent} if the random vector $\bfomega$
converges in distribution to a random vector with two independent
components, i.e., the sum of the input $\inputsum(X_{n})$ and the sum of the output
$\out(X_{n})$ are asymptotically independent random variables.
\end{definition}

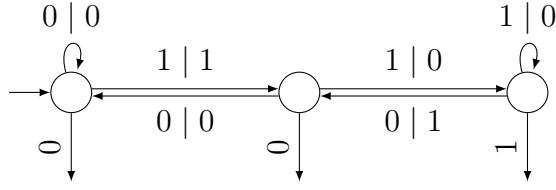
\begin{figure}
\newcommand{\Bold}[1]{\mathbf{#1}}
\begin{tikzpicture}[auto, initial text=, >=latex, accepting text=, accepting/.style=accepting by arrow, accepting distance=5ex, every state/.style={minimum
    size=1.3em}]
\node[state, initial] (v0) at (0.000000,
0.000000) {};
\path[->] (v0.270.00) edge node[rotate=450.00, anchor=south] {$0$} ++(270.00:5ex);
\node[state] (v1) at (3.000000, 0.000000) {};
\path[->] (v1.270.00) edge node[rotate=450.00, anchor=south] {$0$} ++(270.00:5ex);
\node[state] (v2) at (6.000000, 0.000000) {};
\path[->] (v2.270.00) edge node[rotate=450.00, anchor=south] {$1$} ++(270.00:5ex);
\path[->] (v0.10.00) edge node[rotate=0.00, anchor=south] {$1\mid 1$} (v1.170.00);
\path[->] (v1.10.00) edge node[rotate=0.00, anchor=south] {$1\mid 0$} (v2.170.00);
\path[->] (v0) edge[loop above] node {$0\mid 0$} ();
\path[->] (v2.190.00) edge node[rotate=360.00, anchor=north] {$0\mid 1$} (v1.350.00);
\path[->] (v2) edge[loop above] node {$1\mid 0$} ();
\path[->] (v1.190.00) edge node[rotate=360.00, anchor=north] {$0\mid 0$} (v0.350.00);
\end{tikzpicture}
\caption{Transducer to compute the Hamming weight of the non-adjacent form.}
\label{fig:NAF}
\end{figure}

\begin{example}
 In~\cite{Heuberger-Prodinger:2007:hammin-weigh}, Heuberger and Prodinger prove that the Hamming weight of the
standard binary expansion and the Hamming weight of the non-adjacent form are
asymptotically independent. Thus, the transducer in Example~\ref{ex:first} is
independent.\end{example}

\section{Main results}\label{sec:main-results}
In this section, we state the main theorems and corollaries describing
independent transducers and transducers with bounded variance. First, we investigate transducers with bounded variance. Then, we give
an algebraic description and a combinatorial characterization of
independent transducers. 

\subsection{Bounded variance and singular asymptotic vari\-ance-\-co\-vari\-ance matrix}\label{sec:var}

We give a combinatorial characterization of transducers whose
 output sum has asymptotic variance $0$. We also give a combinatorial
 description of transducers with singular asymptotic variance-covariance
 matrix. These characterizations are given in terms of cycles and closed walks
 of directed graphs.

As usual, a \emph{cycle} is a strongly connected digraph such that every
vertex has out-degree $1$. A \emph{closed walk} is an alternating sequence of vertices and edges
  $(s_{1},e_{1},s_{2},\ldots,s_{n+1}=s_{1})$ such that $e_{j}$ is an edge from
  $s_{j}$ to $s_{j+1}$.

For a function $g$ and a
walk $C$ of the underlying graph of the transducer, we define
\begin{equation*}g(C)=\sum_{e\in C}g(e)\end{equation*}
taking multiplicities into account. Here, the function $g$ is either the
constant function $\mathds 1(e)=1$, the input $\eps(e)$ or the
output $\delta(e)$ of the the transition $e$.

\begin{theorem}\label{thm:var}For a subsequential, complete, finally
  connected and finally aperiodic transducer with an arbitrary finite input alphabet $\mathcal A_{I}$, the following assertions are equivalent:
\begin{enumerate}[label=(\alph*)]
\item\label{it:var0} The asymptotic variance $v_{2}$ of the output sum is $0$.
\item \label{it:somewalks} There exists a state $s$ of the
final component and a constant $k\in\R$ such that 
\begin{equation*}\delta(C)=k\mathds 1(C)\end{equation*}
holds for every closed walk $C$ of the final component
  visiting the state $s$ exactly once.
\item\label{it:allcycles} There exists a constant $k\in\R$ such that 
 \begin{equation*}\delta(C)=k\mathds 1(C)\end{equation*}
holds for every directed cycle $C$ of the final component of the transducer
$\T$.
\end{enumerate}
In that case, $kn+\bigOh(1)$ is the expected value of the output sum and
Statement~\itemref{it:somewalks} holds for all states $s$ of the final component.
\end{theorem}

We want to emphasize that only cycles and closed walks of the final component
are considered in this theorem (see also Remark~\ref{rem:non-final-comp}).

In the case of a strongly connected transducer, the equivalent conditions of
Theorem~\ref{thm:var} will be shown to be equivalent to another condition
which, at first glance, seems to be even stronger.

\begin{definition}
The output sum of a transducer is called
\emph{quasi-\-deter\-min\-is\-tic} if there is a constant $k\in\R$ such that 
\begin{equation*}\out(X_{n})=kn+\bigOh(1)\end{equation*}
holds for all $n$ and all inputs. 
\end{definition}

We now characterize quasi-deterministic output sums. In
weakly connected graphs, it turns out that being ``quasi-deterministic'' is a stronger notion
than the conditions in Theorem~\ref{thm:var}.

\begin{theorem}\label{thm:quasi-det}
Let $\T$ be a subsequential, complete transducer whose underlying graph is weakly connected. Then the
following two assertions are equivalent:
\begin{enumerate}[label=(\alph*)]\setcounter{enumi}{3}
\item\label{it:quasidet} There exists a constant $k\in\R$ such that the random variable $\out(X_{n})$ is
  qua\-si-\-deter\-mi\-nis\-tic with value $kn+\bigOh(1)$.
\item\label{it:allcycles-nonfinal} There exists a constant $k\in\R$ such that 
\begin{equation*}\delta(C)=k\mathds 1(C)\end{equation*}
holds for every directed cycle $C$ of the transducer.
\end{enumerate}
\end{theorem}

By comparing statements \itemref{it:allcycles} of Theorem~\ref{thm:var}
and \itemref{it:allcycles-nonfinal} of
Theorem~\ref{thm:quasi-det}, it is obvious that in strongly connected transducers, all these statements are actually equivalent.

\begin{corollary}
Let $\T$ be a subsequential, complete, strongly connected, aperiodic
transducer. Then the asymptotic variance $v_{2}$ of the output sum is zero if
and only if the output sum is a quasi-deterministic random variable.
\end{corollary}

\begin{remark}\label{rem:quasi-det-weak-conn}
If the transducer is not strongly connected (so that there are states that do not belong to the final component), the output sum
can have bounded variance without being quasi-deterministic. A simple example is a transducer that counts the number of 1s in a binary string before the first $0$.
In such a case, however, the transducer formed only by the final component still needs to have quasi-deterministic output sum.
\end{remark}

When considering the special case of the Hamming weight, bounded variance only
occurs in trivial cases:

\begin{corollary}\label{cor:01output}
For $\mathcal A_{O}=\{0,1\}$, the only output weights of the final component with asymptotic variance $v_{2}=0$
are $(0,\ldots,0)$ and $(1,\ldots,1)$.
\end{corollary}

The following corollary of Theorem~\ref{thm:var} gives a combinatorial
characterization of transducers whose asymptotic variance-covariance matrix is singular.

\begin{corollary}\label{thm:reg}
Let $\T$ be a complete, subsequential, finally connected, finally aperiodic
transducer whose input alphabet has at least size $2$. Then the asymptotic
covariance-variance matrix $\Sigma$ has rank $1$ if and only if
there exist $a$, $b\in\R$ with 
\begin{equation}\label{eq:covar-rank1}\delta(C)=a\mathds 1(C)+b\eps(C)\end{equation} 
for all cycles $C$ of the final component.

In that case, the constants are $a=-\frac{c}{v_{1}}e_{1}+e_{2}$ and $b=\frac{c}{v_{1}}$.

Furthermore, the random variables $\inputsum(X_n)$ and $\out(X_n)$ are
asymptotically perfectly positively or negatively correlated (i.e., they have
asymptotic correlation coefficient $\pm 1$) if and only if~\eqref{eq:covar-rank1} holds with $b\neq 0$.
\end{corollary}

\subsection{Algebraic description of independent transducers}\label{sec:alg}
For giving an algebraic description of independent
transducers, we define transition matrices of the transducer.

\begin{definition}\label{def:trans-matrix}
For $\eps\in\mathcal A_{I}$, let a \emph{transition matrix} $M_{\eps}(y)$ of the final component be the $N\times N$-matrix whose entry
$(s,t)$ is $y^{\delta}$ if there is a transition from state $s$ to state $t$
in the final component
with input $\eps$ and output $\delta$, and $0$ otherwise. 

Similarly, let $M_{\eps}'$ be the transition matrix of the whole transducer. The ordering
of the states is considered to be fixed in such a way that the initial state
$1$ is the first state and $M_{\eps}'$ has the
block structure
\begin{equation*}
\begin{pmatrix}*&*\\0&M_{\eps}\end{pmatrix}
\end{equation*}
where $*$ are matrices with arbitrary entries. If the transducer is strongly
connected, the matrices $*$ are not present (they have $0$ rows).
\end{definition}

\begin{theorem}\label{thm:alg}Let $\T$ be a complete, subsequential, finally
  connected, finally aperiodic transducer
  with transition matrices $M_{\eps}(y)$ for $\eps\in\mathcal A_{I}$. Let $K\geq 2$ be the size of
  the input alphabet $\mathcal A_{I}$ and 
\begin{equation*}f(x,y,z)=\det\Big(\mmatrix\Big).\end{equation*}

Then the random variables $\inputsum(X_n)$ and $\out(X_n)$ have the expected
values, variances and covariance
\begin{equation}\label{eq:expectation-and-variances}
\begin{aligned}
\Expect(\inputsum(X_n))&=e_{1}n,\\
\Expect(\out(X_n))&=e_{2}n+\bigOh(1),\\
\Var(\inputsum(X_n))&=v_{1}n,\\
\Var(\out(X_n))&=v_{2}n+\bigOh(1),\\
\Cov(\inputsum(X_n),\out(X_n))&=cn+\bigOh(1)
\end{aligned}
\end{equation}
with
\begin{align*}
e_{1}&=\frac{f_x}{f_z}\Big\vert_{\bfone},\\
e_{2}&=\frac{f_y}{f_z}\Big\vert_{\bfone},\\
v_{1}&=\frac1{f_{z}^{3}}(f_{x}^{2}(f_{zz}+f_{z})+f_{z}^{2}(f_{xx}+f_{x})-2f_{x}f_{z}f_{xz})\Big\vert_{\bfone},\\
v_{2}&=\frac1{f_{z}^{3}}(f_{y}^{2}(f_{zz}+f_{z})+f_{z}^{2}(f_{yy}+f_{y})-2f_{y}f_{z}f_{yz})\Big\vert_{\bfone},\\
c&=
\frac1{f_{z}^{3}}(f_{x}f_{y}(f_{zz}+f_{z})+f_{z}^{2}f_{xy}-f_{y}f_{z}f_{xz}-f_{x}f_{z}f_{yz})\Big\vert_{\bfone}
\end{align*}
where $\bfone=(1,1,1)^{t}$ and $f_{z}(\bfone)\neq 0$.

The constants $e_1$ and $v_1$ can also be expressed as
\begin{align}\label{eq:e-1-v-1-explicit-trivial}
e_{1}&=\frac 1K\sum_{\eps\in\mathcal A_{I}}\eps,&
v_{1}&=\frac 1K\sum_{\eps\in\mathcal A_{I}}\eps^{2}-\Big(\frac 1K\sum_{\eps\in\mathcal
A_{I}}\eps\Big)^{2}.
\end{align}

The random vector $\bfomega$ is
asymptotically jointly normally distributed if and only if the asymptotic
variance-covariance matrix $\Sigma$ is regular.

The transducer $\T$ is independent if and
only if
\begin{align}\label{eq:cov}(f_{x}f_{y}(f_{zz}+f_{z})+f_{z}^{2}f_{xy}-f_{y}f_{z}f_{xz}-f_{x}f_{z}f_{yz})\big\vert_{\bfone}=0
\end{align}
or, equivalently,
\begin{equation}\label{eq:simplified-exp}(e_{1}f_{y}(f_{zz}+f_{z})+f_{z}f_{xy}-f_{y}f_{xz}-e_{1}f_{z}f_{yz})\big\vert_{\bfone}=0.\end{equation}
\end{theorem}

This result has been implemented as the method
\begin{center}\texttt{FiniteStateMachine.asymptotic\_moments()}\end{center} in the computer
algebra system \textsf{Sage}, cf.\ \cite{Heuberger-Kropf:trac-16145}, using the finite state machines package described in
\cite{Heuberger-Krenn-Kropf:ta:finit-state}.

\begin{remark}\label{rem:non-final-comp}
Neither the final output nor the non-final components influence the asymptotic
result because it only depends on $f(x,y,z)$ and thus on the transitions of
the final component.
\end{remark}

Now we consider the following ``inverse'' problem: Given the underlying graph and the
input digits of the transducer; how can we choose the output labels such that
the transducer is independent? 

Let $(a_{1},\ldots,a_{KN})$ be the output labels of the final component
of the transducer. We say, as usual, that a linear equation is homogeneous if
the zero vector is a solution. Then~\eqref{eq:cov} is a linear, homogeneous equation in
$a_{1},\ldots,a_{KN}$ with real coefficients. The equation is linear because the variables
$a_{i}$ only occur linearly in the exponents of $y$ and there are
only first derivatives with respect to $y$ in the covariance
condition~\eqref{eq:cov}.  Furthermore,~\eqref{eq:cov} is homogeneous because all derivatives with
respect to $y$ (and maybe other additional variables) at $(x,y,z)^{t}=\bfone$ are homogeneous. 
A solution of this linear, homogeneous equation corresponds
to an independent transducer. 

Let us first consider the situation where all outputs are equal to $1$. Then, the determinant $f(x,y,z)$ consists of monomials $x^{a}y^{b}z^{b}$ with $a\in\R$ and
$b\in\Z$. Therefore, we obtain
\begin{align*}
f_{y}\vert_{\bfone}&=f_{z}\vert_{\bfone},\\
f_{xy}\vert_{\bfone}&=f_{xz}\vert_{\bfone},\\
f_{yz}\vert_{\bfone}&=f_{zz}+f_{z}\vert_{\bfone},
\end{align*}
and it follows that~\eqref{eq:cov} and \eqref{eq:simplified-exp} are satisfied. This means that a constant output $(k,\ldots,k)$ for $k\in\mathcal A_{O}$ is always a trivial solution to these equations because~\eqref{eq:cov} is homogeneous.

But for these trivial solutions, the sum of the output is an asymptotically
degenerate random variable. Hence, we are not really interested in the independent transducers given by these
solutions.

\begin{example}\label{ex:simple}
In Figure~\ref{fig:exsimple}, we have a transducer with variable output
weights $a_{1}$, $a_{2}$, $a_{3}$ and $a_{4}$. We do not give the final output labels as they do not influence the asymptotic
result. In this example,~\eqref{eq:cov}
simplifies to
\begin{equation*}-a_{1}+a_{2}=0.\end{equation*}

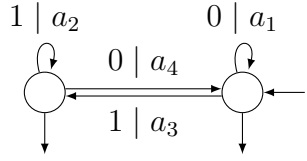
\begin{figure}
\newcommand{\Bold}[1]{\mathbf{#1}}
\begin{tikzpicture}[auto, initial text=, >=latex, accepting text=,
  accepting/.style=accepting by arrow, every state/.style={minimum
    size=1.3em}]
\node[state, accepting, accepting where=below, initial, initial where=right] (v0) at (1.300000, 0.000000) {};
\node[state, accepting, accepting where=below] (v1) at (-1.300000, 0.000000) {};
\path[->] (v0.190.00) edge node[rotate=360.00, anchor=north] {$1\mid a_{3}$} (v1.350.00);
\path[->] (v1.10.00) edge node[rotate=0.00, anchor=south] {$0\mid a_{4}$} (v0.170.00);
\path[->] (v0) edge[loop above] node {$0\mid a_{1}$} ();
\path[->] (v1) edge[loop above] node {$1\mid a_{2}$} ();
\end{tikzpicture}
\caption{Transducer of Example~\ref{ex:simple}.}
\label{fig:exsimple}
\end{figure}
\end{example}

\subsection{Combinatorial characterization of independent transducers}\label{sec:comb}
We connect the derivatives of $f(x,y,z)$ with a weighted sum
of subgraphs of the underlying graph. Thus, in Theorem~\ref{thm:comb}, we can give a combinatorial
description of~\eqref{eq:cov}. 
\begin{definition}We define the following types of directed graphs as subgraphs
  of the final component of the transducer.
\begin{itemize}
\item A \emph{rooted tree} is a weakly connected digraph with one vertex which has out-degree
$0$, while all other vertices have out-degree $1$. The vertex with out-degree $0$
is called the \emph{root} of the tree.

\item A \emph{functional digraph} is a digraph whose vertices have out-degree
$1$. Each component of a
functional digraph consists of a directed cycle and some trees rooted at
vertices of the cycle. For a functional digraph $D$, let $\mathcal C_{D}$ be the
set of all cycles of $D$.\end{itemize}
\end{definition}

\begin{definition}\label{def:sum-over-graphs}
Let $\mathcal D_{1}$ and $\mathcal D_{2}$ be the sets of all spanning
subgraphs of the final component of the
transducer $\T$ which are functional digraphs and have one and two
components, respectively.

For functions $g$ and $h\colon\mathcal E\to \R$, we define
\begin{align*}
g(\mathcal D_{1})&=\sum_{D\in\mathcal D_{1}}\sum_{C\in\mathcal
  C_{D}}g(C),\\
gh(\mathcal D_{1})&=\sum_{D\in\mathcal D_{1}}\sum_{C\in\mathcal
C_{D}}g(C)h(C),\\
gh(\mathcal D_{2})&=\sum_{D\in\mathcal D_{2}}\sum_{C_{1}\in\mathcal
C_{D}}\sum_{\substack{C_{2}\in\mathcal C_{D}\\C_{2}\neq
  C_{1}}}g(C_{1})h(C_{2}).
\end{align*}
\end{definition}

With these definitions, we give a combinatorial characterization of independent transducers.

\begin{theorem}\label{thm:comb}Let $\T$ be a complete, subsequential, finally
  connected, finally aperiodic transducer. 
  
  Then the random variables $\inputsum(X_n)$ and $\out(X_n)$ have the expected
  values given by \eqref{eq:expectation-and-variances}, where the constants are
\begin{align*}
e_{1}&=\frac{\eps(\mathcal D_1)}{\mathds 1(\mathcal D_1)},\\
e_{2}&=\frac{\delta(\mathcal D_1)}{\mathds 1(\mathcal
  D_1)}.
\end{align*}
The
variances and the covariance are given by
\eqref{eq:expectation-and-variances}, with the constants
\begin{align*}
v_{1}&=\frac{1}{\mathds 1(\mathcal D_{1})}\big((\eps-e_{1}\mathds
1)(\eps-e_{1}\mathds 1)(\mathcal D_{1})-(\eps-e_{1}\mathds 1)(\eps-e_{1}\mathds 1)(\mathcal D_{2})\big),\\
v_{2}&=\frac{1}{\mathds 1(\mathcal D_{1})}\big((\delta-e_{2}\mathds
1)(\delta-e_{2}\mathds 1)(\mathcal D_{1})-(\delta-e_{2}\mathds 1)(\delta-e_{2}\mathds 1)(\mathcal D_{2})\big),\\
c&=\frac{1}{\mathds 1(\mathcal D_{1})}\big((\eps-e_{1}\mathds
1)(\delta-e_{2}\mathds 1)(\mathcal D_{1})-(\eps-e_{1}\mathds 1)(\delta-e_{2}\mathds 1)(\mathcal D_{2})\big).
\end{align*}
  
  The transducer $\T$ is independent if and only if
\begin{align}\label{eq:comb}
(\eps-e_{1}\mathds
1)(\delta-e_{2}\mathds 1)(\mathcal D_{1})=(\eps-e_{1}\mathds 1)(\delta-e_{2}\mathds 1)(\mathcal D_{2})
\end{align}
\end{theorem}

We emphasize that, by Definition~\ref{def:sum-over-graphs}, only edges in the
final component of the transducer are considered in
Theorem~\ref{thm:comb}. The non-final components do not influence the
asymptotic main terms (see also Remark~\ref{rem:non-final-comp}).

In the following corollary, we consider the case of a normalized input and
output, i.e., the constants of the expected values satisfy
$e_{1}=e_{2}=0$. This can be obtained by subtracting the original constants
$e_{1}$ and $e_{2}$ from every input label and output label,
respectively. Then the
corollary follows directly from Theorem~\ref{thm:comb}.

\begin{corollary}\label{cor:comb-indep}
Suppose that $\Expect (\inputsum(X_{n}))$ and $\Expect(\out(X_{n}))$ are both bounded. Then the transducer $\T$ is independent if and only if
\begin{equation*}\eps\delta(\mathcal D_{2})=\eps\delta(\mathcal D_{1}).\end{equation*}
\end{corollary}

\begin{example}\label{ex:simple2}
We again consider the transducer of Example~\ref{ex:simple} in
Figure~\ref{fig:exsimple}. The set $\mathcal D_{1}$ consists of $3$ functional
digraphs and $\mathcal D_{2}$ consists of only one functional digraph (see
Figure~\ref{fig:funcdigraph}). By~\eqref{eq:comb}, we obtain the same
equation as before, namely
\begin{equation*}a_{1}-a_{2}=0,\end{equation*}
as condition for the transducer to be independent.

Also by Theorem~\ref{thm:comb}, the expected value of the output sum is
\begin{equation*}\frac{a_{1}+a_{2}+a_{3}+a_{4}}4n+\bigOh(1)\end{equation*}
and the asymptotic  variance is
\begin{equation*}
  \frac{5a_1^2 - 6a_1a_2 + 5a_2^2 - 2a_1a_3 - 2a_2a_3 +
a_3^2 - 2a_1a_4 - 2a_2a_4 + 2a_3a_4 + a_4^2}{16}.
\end{equation*}
The covariance between the input sum and the output sum is
\begin{equation*}
  -\frac{a_1 - a_2}{4}n +\bigOh(1).
\end{equation*}
\begin{figure}
\centering
\begin{subfigure}[b]{\textwidth}\centering
\newcommand{\Bold}[1]{\mathbf{#1}}
\begin{tikzpicture}[auto, initial text=, >=latex, accepting text=,
  accepting/.style=accepting by arrow, every state/.style={minimum
    size=1.3em}]
\useasboundingbox (-1.65, -0.65) rectangle (1.65, 1.5);
\node[state] (v0) at (1.200000, 0.000000) {};
\node[state] (v1) at (-1.200000, 0.000000) {};
\path[->, line width=0.8pt] (v0.190.00) edge node[rotate=360.00, anchor=north] {$1\mid a_{3}$} (v1.350.00);
\path[->, color=gray] (v1.10.00) edge node[rotate=0.00, anchor=south] {$\ $} (v0.170.00);
\path[->, color=gray] (v0) edge[loop above] node {} ();
\path[->, line width=0.8pt] (v1) edge[loop above] node {$1\mid a_{2}$} ();
\end{tikzpicture}\quad
\begin{tikzpicture}[auto, initial text=, >=latex, accepting text=,
  accepting/.style=accepting by arrow, every state/.style={minimum
    size=1.3em}]
\useasboundingbox (-1.65, -0.65) rectangle (1.65, 1.5);
\node[state] (v0) at (1.200000, 0.000000) {};
\node[state] (v1) at (-1.200000, 0.000000) {};
\path[->, line width=0.8pt] (v0.190.00) edge node[rotate=360.00, anchor=north] {$1\mid a_{3}$} (v1.350.00);
\path[->, line width=0.8pt] (v1.10.00) edge node[rotate=0.00, anchor=south] {$0\mid a_{4}$} (v0.170.00);
\path[->, color=gray] (v0) edge[loop above] node {} ();
\path[->, color=gray] (v1) edge[loop above] node {} ();
\end{tikzpicture}\quad
\begin{tikzpicture}[auto, initial text=, >=latex, accepting text=,
  accepting/.style=accepting by arrow, every state/.style={minimum
    size=1.3em}]
\useasboundingbox (-1.65, -0.65) rectangle (1.65, 1.5);
\node[state] (v0) at (1.200000, 0.000000) {};
\node[state] (v1) at (-1.200000, 0.000000) {};
\path[->, color=gray] (v0.190.00) edge node[rotate=360.00, anchor=north] {} (v1.350.00);
\path[->, line width=0.8pt] (v1.10.00) edge node[rotate=0.00, anchor=south] {$0\mid a_{4}$} (v0.170.00);
\path[->, line width=0.8pt] (v0) edge[loop above] node {$0\mid a_{1}$} ();
\path[->, color=gray] (v1) edge[loop above] node {} ();
\end{tikzpicture}
\caption{$\mathcal D_{1}$}
\end{subfigure}\\
\begin{subfigure}[b]{\textwidth}\centering
\newcommand{\Bold}[1]{\mathbf{#1}}\begin{tikzpicture}[auto, initial text=, >=latex, accepting text=, accepting/.style=accepting by arrow, every state/.style={minimum size=1.3em}]
\node[state] (v0) at (1.200000, 0.000000) {};
\node[state] (v1) at (-1.200000, 0.000000) {};
\path[->, color=gray] (v0.190.00) edge node[rotate=360.00, anchor=north] {} (v1.350.00);
\path[->, color=gray] (v1.10.00) edge node[rotate=0.00, anchor=south] {} (v0.170.00);
\path[->, line width=0.8pt] (v0) edge[loop above] node {$0\mid a_{1}$} ();
\path[->, line width=0.8pt] (v1) edge[loop above] node {$1\mid a_{2}$} ();
\end{tikzpicture}
\caption{$\mathcal D_{2}$}
\end{subfigure}
\caption{Functional digraphs of the transducer of Example~\ref{ex:simple2}.}
\label{fig:funcdigraph}
\end{figure}
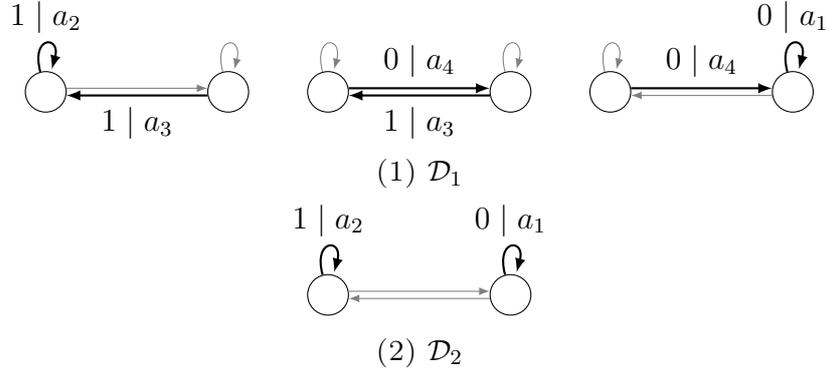
\end{example}

\section{Examples of transducers}\label{sec:ex}
In this section we give various examples to illustrate our theorems: these include both dependent and independent transducers and transducers with both bounded and unbounded
variance of the output sum. These examples are also shown in the
documentation of the method \texttt{FiniteStateMachine.asymptotic\_moments()}
\cite{Heuberger-Kropf:trac-16145} in \textsf{Sage}. Example~\ref{ex:tau-adic}
demonstrates how the combinatorial characterization of transducers with
bounded variance can be  used in cases where we only have limited information
about the transducer.

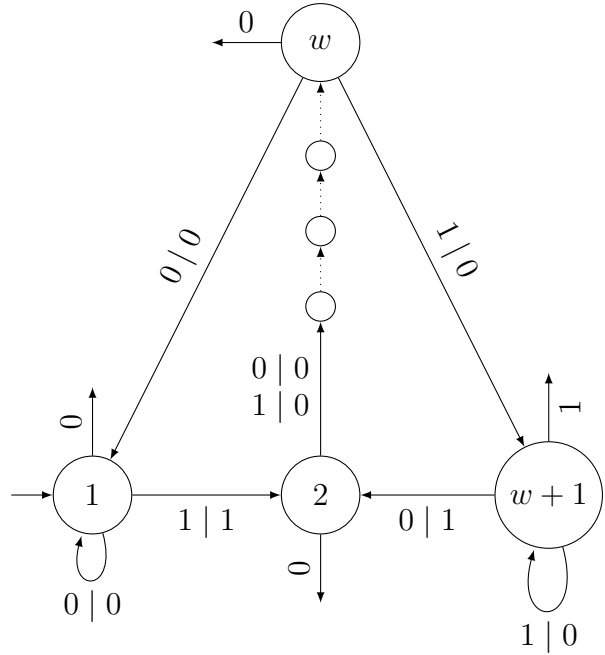
\begin{figure}
\newcommand{\Bold}[1]{\mathbf{#1}}\begin{tikzpicture}[auto, initial text=, >=latex, accepting text=, accepting/.style=accepting by arrow]
\node[state, initial] (v0) at (0.000000,
0.000000) {$1$};
\path[->] (v0.90.00) edge node[rotate=270.00, anchor=north] {$0$} ++(90.00:5ex);
\node[state] (v1) at (3.000000, 0.000000) {$2$};
\path[->] (v1.270.00) edge node[rotate=450.00, anchor=south] {$0$} ++(270.00:5ex);
\node[state] (v2) at (6.000000, 0.000000) {$w+1$};
\path[->] (v2.90.00) edge node[rotate=90.00, anchor=north] {$1$} ++(90.00:5ex);
\node[state] (v3) at (3, 6) {$w$};
\path[->] (v3.180.00) edge node[rotate=0.00, anchor=south] {$0$}
++(180.00:5ex);
\node[state, minimum size=0em] (v4) at (3, 2.5) {};
\node[state, minimum size=0em] (v5) at (3, 3.5) {};
\node[state, minimum size=0em] (v6) at (3, 4.5) {};
\path[->] (v0) edge node[rotate=0.00, anchor=north] {$1\mid 1$} (v1);
\path[->] (v1) edge node[auto, align=center] {$0\mid
  0$\\$1\mid 0$} (v4);
\path[->] (v4) edge[dotted] node {} (v5);
\path[->] (v5) edge[dotted] node {} (v6);
\path[->] (v6) edge[dotted] node {} (v3);
\path[->] (v0) edge[loop below] node {$0\mid 0$} ();
\path[->] (v2) edge node[rotate=360.00, anchor=north] {$0\mid 1$} (v1);
\path[->] (v2) edge[loop below] node {$1\mid 0$} ();
\path[->] (v3) edge node[rotate=63.43, anchor=south] {$0\mid 0$} (v0);
\path[->] (v3) edge node[rotate=-63.43, anchor=south] {$1\mid 0$} (v2);
\end{tikzpicture}
\caption{Transducer to compute the Hamming weight of the width-$w$ non-adjacent form.}
\label{fig:wNAF}
\end{figure}

\begin{example}[Width-$w$ non-adjacent form]\label{ex:wNAF}
The width-$w$ non-ad\-jacent form (cf.~\cite{avanzi:mywnaf,muirstinson:minimality}) is a digit expansion with base
$2$, digits $\{0,\allowbreak \pm1,\allowbreak \pm3,\allowbreak \ldots,\allowbreak \pm(2^{w-1}-1)\}$ and the syntactical rule
that at most one of any $w$ consecutive digits is nonzero. The transducer in
Figure~\ref{fig:wNAF} computes the Hamming weight of the width-$w$
non-adjacent form when reading the standard binary expansion
(cf.~\cite{Heuberger-Kropf:2013:analy}). For $w=2$, this transducer is the
same as that in Figure~\ref{fig:NAF}. The variance of the output is not $0$ (Corollary~\ref{cor:01output}). With Theorem~\ref{thm:alg} or~\ref{thm:comb}, we obtain
that this transducer is independent for every $w$. Thus, the Hamming weight of the width-$w$
non-adjacent form and the standard binary expansion are asymptotically
independent. 
\end{example}

\begin{remark}Example~\ref{ex:wNAF} not only shows that there are infinitely
  many independent transducers, but also gives the construction of one such infinite
  family of independent transducers.
\end{remark}

\begin{example}[Gray code]
The Gray code is an encoding of the positive integers such that the Gray code
of $n$ and the Gray code of $n+1$ differ only at one position. The transducer in Figure~\ref{fig:gray} computes the Gray
code of an integer. The output label of the initial state is $0$ and, as it
does not influence the result, it is not given in the figure. The transducer is finally connected and finally aperiodic. The final component consisting of states $2$ and $3$ is
independent (see Example~\ref{ex:simple}). Thus, the Hamming weight of the
Gray code and the standard binary expansion are asymptotically independent.
\begin{figure}
\newcommand{\Bold}[1]{\mathbf{#1}}\begin{tikzpicture}[auto, initial text=, >=latex, accepting text=, accepting/.style=accepting by arrow, accepting distance=5ex]
\node[state, initial] (v0) at (0.000000, 0.000000) {$1$};
\node[state] (v1) at (3.000000, 2.000000) {$2$};
\path[->] (v1.0.00) edge node[rotate=0.00, anchor=south] {$0$} ++(0.00:5ex);
\node[state] (v2) at (3.000000, -2.000000) {$3$};
\path[->] (v2.0.00) edge node[rotate=0.00, anchor=south] {$1$} ++(0.00:5ex);
\path[->] (v0) edge node[rotate=33.69, anchor=south] {$0\mid -$} (v1);
\path[->] (v1.-85.00) edge node[rotate=90.00, anchor=north] {$1\mid 1$} (v2.85.00);
\path[->] (v2.95.00) edge node[rotate=90.00, anchor=south] {$0\mid 1$} (v1.265.00);
\path[->] (v1) edge[loop above] node {$0\mid 0$} ();
\path[->] (v2) edge[loop below] node {$1\mid 0$} ();
\path[->] (v0) edge node[rotate=-33.69, anchor=north] {$1\mid -$} (v2);
\end{tikzpicture}
\caption{Transducer to compute the Gray code.}
\label{fig:gray}
\end{figure}
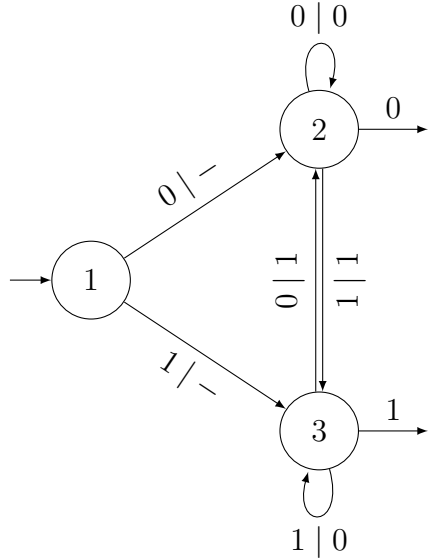
\end{example}

\begin{example}[Length $2$ blocks in the standard binary expansion] 
We count the number of patterns of length $2$ occurring in the standard binary
expansion and compare it to the Hamming weight. By symmetry, it is obviously sufficient
to consider the two patterns $01$ and $11$. The transducers in Figure~\ref{fig:11}
determine the number of $01$- and $11$-blocks, respectively. The variance of the
output weight is not $0$ in either case (Corollary~\ref{cor:01output}), in fact the
constant $v_2$ is $\frac{1}{16}$ (for $01$-blocks) and $\frac{5}{16}$ respectively.

By Theorem~\ref{thm:alg} or~\ref{thm:comb}, we also find that the transducer for
$01$-blocks is independent, while the transducer for $11$-blocks (unsurprisingly) is not:
the number of $11$-blocks asymptotically depends on the number of $1$'s in the
standard binary expansion, and the correlation coefficient is $\frac2{\sqrt{5}}\approx 0.894$.
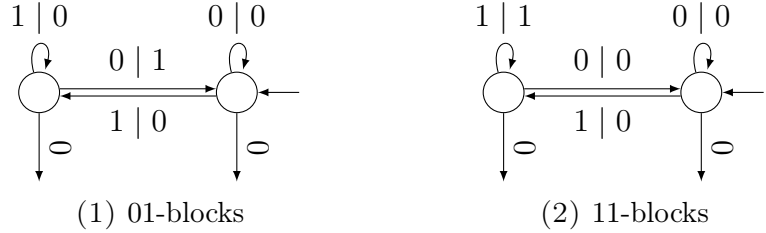
\begin{figure}
\centering
\begin{subfigure}[b]{0.45\textwidth}\centering
\newcommand{\Bold}[1]{\mathbf{#1}}\begin{tikzpicture}[auto, initial text=,
  >=latex, accepting text=, accepting/.style=accepting by arrow,  every state/.style={minimum size=1.3em}]
\node[state, initial, initial where=right] (v0) at (1.300000, 0.000000) {};
\path[->] (v0.270.00) edge node[rotate=-90.00, anchor=south] {$0$} ++(270.00:5ex);
\node[state] (v1) at (-1.300000, 0.000000) {};
\path[->] (v1.270.00) edge node[rotate=-90.00, anchor=south] {$0$} ++(270.00:5ex);
\path[->] (v1.10.00) edge node[rotate=0.00, anchor=south] {$0\mid 1$} (v0.170.00);
\path[->] (v0) edge[loop above] node {$0\mid 0$} ();
\path[->] (v0.190.00) edge node[rotate=360.00, anchor=north] {$1\mid 0$} (v1.350.00);
\path[->] (v1) edge[loop above] node {$1\mid 0$} ();
\end{tikzpicture}
\caption{$01$-blocks}
\end{subfigure}\quad
\begin{subfigure}[b]{0.45\textwidth}\centering

\newcommand{\Bold}[1]{\mathbf{#1}}\begin{tikzpicture}[auto, initial text=,
  >=latex, accepting text=, accepting/.style=accepting by arrow, every state/.style={minimum size=1.3em}]
\node[state, initial, initial where=right] (v0) at (1.300000, 0.000000) {};
\path[->] (v0.270.00) edge node[rotate=-90.00, anchor=south] {$0$} ++(270.00:5ex);
\node[state] (v1) at (-1.300000, 0.000000) {};
\path[->] (v1.270.00) edge node[rotate=-90.00, anchor=south] {$0$} ++(270.00:5ex);
\path[->] (v1.10.00) edge node[rotate=0.00, anchor=south] {$0\mid 0$} (v0.170.00);
\path[->] (v0) edge[loop above] node {$0\mid 0$} ();
\path[->] (v0.190.00) edge node[rotate=360.00, anchor=north] {$1\mid 0$} (v1.350.00);
\path[->] (v1) edge[loop above] node {$1\mid 1$} ();
\end{tikzpicture}
\caption{$11$-blocks}
\end{subfigure}
\caption{Transducers to count the number of $01$- and $11$-blocks in the standard binary expansion.}
\label{fig:11}
\end{figure}
\end{example}

\begin{example}
Now, we give an example of a transducer with bounded variance of the output sum. We compute
the number of $10$-blocks minus the number of $01$-blocks in the standard
binary digit expansion. In Figure~\ref{fig:0110blocks}, we show the
corresponding transducer. The output label of the initial state is $0$ and, as it
does not influence the result, it is not given in the figure. Any of the three cycles has output sum $0$. Thus, the
asymptotic variance of this random variable is $0$. There is, of course, an intuitive explanation: when we read
a $1$ after a $0$ (reading from right to left), the count increases by $1$; when we read a $0$ after a $1$, the
count decreases by $1$; otherwise, it remains unchanged. Thus the final output value will only depend on the
first and last digit.

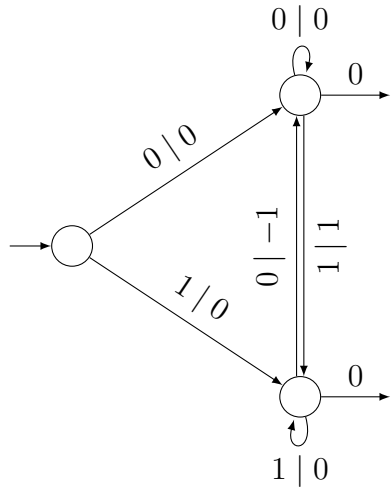
\begin{figure}
\newcommand{\Bold}[1]{\mathbf{#1}}\begin{tikzpicture}[auto, initial text=,
  >=latex, accepting text=, accepting/.style=accepting by arrow, every state/.style={minimum size=1.3em}]
\node[state] (v0) at (3.000000, -2.000000) {};
\node[state] (v1) at (3.000000, 2.000000) {};
\node[state, initial] (v2) at (0.000000, 0.000000) {};
\path[->] (v0.0.00) edge node[rotate=0.00, anchor=south] {$0$} ++(0.00:5ex);
\path[->] (v1.0.00) edge node[rotate=0.00, anchor=south] {$0$} ++(0.00:5ex);
\path[->] (v0.100.00) edge node[rotate=90.00, anchor=south] {$0\mid -1$} (v1.260.00);
\path[->] (v0) edge[loop below] node {$1\mid 0$} ();
\path[->] (v2) edge node[rotate=33.69, anchor=south] {$0\mid 0$} (v1);
\path[->] (v1) edge[loop above] node {$0\mid 0$} ();
\path[->] (v2) edge node[rotate=-33.69, anchor=south] {$1\mid 0$} (v0);
\path[->] (v1.-80.00) edge node[rotate=90.00, anchor=north] {$1\mid 1$} (v0.80.00);
\end{tikzpicture}
\caption{Transducer to compute the number of $10$-blocks minus the number of $01$-blocks in the standard binary expansion.}
\label{fig:0110blocks}
\end{figure}
\end{example}

\begin{example}\label{ex:tau-adic}
Finally, we consider the transducer used in
\cite{Heigl-Heuberger:2012:analy-digit} to compute the minimal Hamming weight of the
$\tau$-adic digit expansion for a given algebraic integer $\tau$ and a given digit set $\mathcal
D$. Note that the output alphabet of the transducer need not be $\{0,1\}$ even if we are
interested in the Hamming weight. The next theorem is an
extension of Theorem~4 in~\cite{Heigl-Heuberger:2012:analy-digit}.
\begin{theorem}Assume that
  $\mathcal D\subset\Z[\tau]^{d}$, for $d$ a positive integer, and $\mathcal
  D\cap\tau\Z^{d}=\{0\}$. Let $\mw(z)$ be the minimal Hamming weight of a
  $\tau$-adic joint digit representation of $z$ with digits in $\mathcal D$. Assume
  further that
  the digit set $\mathcal D$ satisfies
\begin{equation*}\forall c\in\Z[\tau]^{d}\quad\exists U\in\mathbb R\quad\forall z\in\Z[\tau]^{d}:\,
  |\mw(z+c)-\mw(z)|\leq U.
\end{equation*}
 Consider the random variable $W_{n}=\mw(D_{n})$, where $D_{n}$ is a random
 $\tau$-adic joint digit representation of length $n$ with digits in
 $\mathcal A_{I}\subset\Z[\tau]^{d}$. We assume that
 $(\tau,\mathcal A_{I})$ is an irredundant digit system with $0\in\mathcal A_{I}$. The digits of $D_{n}$ are independent and identically
 distributed with uniform distribution on $\mathcal A_{I}$. 

Then there exist constants $E$, $V$, with $V\neq0$, such that
\begin{align*}\Expect W_{n}&=En+\bigOh(1),\\
\Var W_{n}&=Vn+\bigOh(1)
\end{align*}
and 
\begin{equation*}\frac{W_{n}-En}{\sqrt{Vn}}\end{equation*}
 is asymptotically normally distributed.
\end{theorem}
\begin{proof}In~\cite{Heigl-Heuberger:2012:analy-digit}, the authors give a
  strongly connected and aperiodic
  transducer computing $\mw(z)$ if the input is the $\tau$-adic representation
  of $z$ with digit set $\mathcal A_{I}$ read from left to right. Everything follows from Theorem 4
  in~\cite{Heigl-Heuberger:2012:analy-digit} if $V\neq0$.

  To prove $V\neq 0$, we use Theorem~\ref{thm:var},
 \itemref{it:somewalks}. In~\cite{Heigl-Heuberger:2012:analy-digit}, the authors state
  that the transducer has a loop at the initial state $1$ with input and output
  digit $0$. Thus, in Theorem~\ref{thm:var}, 
  \itemref{it:somewalks}, the value of $k$ is $0$.

On the other hand, there exists a $z\in\Z[\tau]^{d}$ with $\mw(z)\neq0$. The input
$z$ leads to a state $s$. From each
state the input $0^{l}$, for some $l$, leads again to the initial state $1$. Thus,
the unique path whose input labels are given by the digit representation of $z\tau^{l}$ is a
closed walk visiting $1$ at least once. The output sum of this closed walk is
$\mw(z\tau^{l})=\mw(z)\neq 0$. Thus, there exists a closed walk whose output
sum is not $0$, which contradicts Theorem~\ref{thm:var}, 
  \itemref{it:somewalks} with $k=0$. Therefore, we obtain $V\neq 0$.
\end{proof}
\end{example}

\section{Proofs of the theorems}\label{sec:proofs-theorems}
In this section, we give the proofs of the theorems and corollaries of
Section~\ref{sec:main-results}. We first prove the algebraic description and
the combinatorial characterization in Sections~\ref{sec:alg} and \ref{sec:comb}. Later we prove the statements in Section~\ref{sec:var} about the
bounded variance.

\subsection{Algebraic description of independent transducers}
First, we prove a slight extension of the $2$-dimensional
Quasi-Power Theorem~\cite{Heuberger:2007:quasi-power} (a generalization
of~\cite{Hwang:1998}). This extension will also take into account the case of
a singular Hessian matrix.

We write boldface letters for a vector $\bfs=(s_{1},s_{2})^{t}$. Furthermore, we use the notation 
$e^{\bfs}=(e^{s_{1}},e^{s_{2}})$. We denote by $\bfone$ a $2$- or
$3$-dimensional vector of ones, depending on the context. By $\|\cdot\|$, we
denote the maximum norm $\|\bfs\|=\max(\lvert s_{1}\rvert,\lvert s_{2}\rvert)$.

\begin{theorem}\label{thm:quasi}
Let $(\bfomega)_{n\geq1}$ be a sequence of $2$-dimensional real random
vectors. Suppose that the moment generating function satisfies 
\begin{equation*}\Expect(e^{\langle\bfomega,\bfs\rangle})=e^{u(\bfs)\Phi(n)+v(\bfs)}\big(1+\bigOh(\kappa_{n}^{-1})\big),\end{equation*}
the $\bigOh$-term being uniform for $\|\bfs\|\leq\tau$,
$\bfs\in\mathbb C^{2}$, $\tau>0$, where
\begin{enumerate}
\item $u(\bfs)$ and $v(\bfs)$ are analytic for $\|s\|\leq\tau$ and
  independent of $n$;
\item $\lim_{n\rightarrow\infty}\Phi(n)=\infty$;
\item $\lim_{n\rightarrow\infty}\kappa_{n}=\infty$.
\end{enumerate}

Then, 
\begin{equation}\label{eq:quasi-power-expectation-variance}
\begin{aligned}
  \Expect(\bfomega)&= \Phi(n)\grad u(\bfzero)+\grad v(\bfzero)+\bigOh(\kappa_n^{-1}),\\
  \Var(\bfomega)&=\Phi(n)H_{u}(\bfzero)+H_v(\bfzero)+\bigOh(\kappa_n^{-1}),
\end{aligned}
\end{equation}
where $H_u(\bfs)$ is the Hessian matrix of $u$. Let $\Sigma$ be the matrix $H_{u}(\bfzero)$.

If $H_{u}(\bfzero)$ is regular, then the standardized random vector 
\begin{equation*}\bfomega^{*}=\frac{\bfomega-\Phi(n)\grad u(\bfzero)}{\sqrt{\Phi(n)}}\end{equation*}
is asymptotically jointly
normally distributed with variance-covariance matrix $\Sigma$.

If $H_{u}(\bfzero)$ has rank $1$, then the limit distribution of $\bfomega^{*}$ is
the direct product of a normal distribution and a degenerate distribution (if one of the
variances is $\bigOh(1)$) or a linear transformation thereof. In the
first case, the coordinates of $\bfomega^{*}$ are asymptotically independent. In the second case,
we have an asymptotically linear relationship between the two coordinates.

If $H_{u}(\bfzero)$ has rank $0$, then the limit distribution of
$\bfomega^{*}$ is degenerate.
\end{theorem}

\begin{proof}
The expressions \eqref{eq:quasi-power-expectation-variance} for expectation and
variance-covariance matrix follow from the moment generating function by differentiation.

The case of a regular Hessian matrix $H_{u}(\bfzero)$ is exactly the statement of
the $2$-dimensional Quasi-Power Theorem~\cite{Heuberger:2007:quasi-power}.

For the case of a singular Hessian matrix, we follow the proof of the Quasi-Power
Theorem~\cite{Heuberger:2007:quasi-power}. We consider the characteristic
function 
\begin{equation*}f_{n}(\bfs)=\exp\Big(-\frac12\bfs^{t}H_{u}(\bfzero)\bfs+\bigOh\Big(\frac{\|s\|^3+\|s\|}{\sqrt{\Phi(n)}}\Big)\Big)\big(1+\bigOh(\kappa_{n}^{-1})\big)\end{equation*}
of the standardized random
vector $\bfomega^{*}$.
Thus the characteristic function tends to
\begin{equation*}
  f(\bfs)=\exp\Big(-\frac12\bfs^{t}H_{u}(\bfzero)\bfs\Big).
\end{equation*}

If the Hessian matrix $H_{u}(\bfzero)$ has rank $0$, then
$f(\bfs)$ equals the identity function. Thus, the distribution function is
degenerate.

If the Hessian matrix $H_{u}(\bfzero)$ has rank $1$ and the variance of the
second coordinate $\Omega_{n,2}$ is $\bigOh(1)$, then
$H_{u}(\bfzero)=\bigl(\begin{smallmatrix}v_{1}&0\\0&0\end{smallmatrix}\bigr)$ for a $v_{1}\in\R$. Thus,
\begin{equation*}f(\bfs)=\exp\Big(-\frac12v_{1}^{2}s_{1}^{2}\Big)\cdot 1\end{equation*}
which is the characteristic function of the normal distribution with mean
$0$ and variance $v_{1}$ times the
characteristic function of the point mass at $0$.

If the Hessian matrix
$H_{u}(\bfzero)=\bigl(\begin{smallmatrix}v_{1}&c\\c&v_{2}\end{smallmatrix}\bigr)$
has rank $1$ with $v_{1}v_{2}\neq 0$, then we consider the random variables $X=\Omega_{n,1}$,
the first coordinate of $\bfomega$, and
$Z=-\frac{c}{v_{1}}\Omega_{n,1}+\Omega_{n,2}$. Then, the main term of the
variance-covariance matrix of $(X,Z)^{t}$ is
$\bigl(\begin{smallmatrix}v_{1}&0\\0&0\end{smallmatrix}\bigr)\Phi(n)$. Thus, $X$ is asymptotically normally
distributed and $Z$ is an asymptotically constant random variable (see
previous case).
\end{proof}

Using this version of the Quasi-Power Theorem, we prove the algebraic
description of independent transducers given in Theorem~\ref{thm:alg}.

\begin{proof}[Proof of Theorem~\ref{thm:alg}]
Let $a_{kln}$ be the number of sequences
of length $n$ with input sum $k$ such that
the corresponding output of the transducer $\T$ has sum $l$. We define
\[A(x,y,z)=\sum_{k\in\R}\sum_{l\in\R}\sum_{n=0}^{\infty}a_{kln}K^{-n}x^{k}y^{l}z^{n}.\]
Thus, the variable $x$
marks the input sum, $y$ marks the output sum, and $z$ marks the length
of the input. Then
$[z^n]A(x,y,z)$ is the probability
generating function of $\bfomega$, where $[z^{n}]b(z)$ is the coefficient of $z^{n}$ in the power series
$b(z)$.

Due to the block structure of $M_{\eps}'(y)$, we have
\begin{equation}\label{eq:A-frac}
\begin{aligned}A(x,y,z)&=\boldsymbol{u}^{t}\Big(I-\frac zK\sum_{\eps\in\mathcal A_{I}}x^{\eps}M_{\eps}'(y)\Big)^{-1}\boldsymbol{v}\\
&=\frac{F_{1}(x,y,z)}{F_{2}(x,y,z)\det\big(\mmatrix\big)},
\end{aligned}
\end{equation}

with $\boldsymbol{u}^{t}=(1,0,\ldots,0)$ for the initial state, $v_{s}=y^{a(s)}$ for the
final output label at state $s$ and $F_{1}(x,y,z)$ and $F_{2}(x,y,z)$ ``polynomials'' in $x$, $y$
and $z$. We use quotation marks because exponents of $x$ and $y$ might not be
integers. However, only finitely many summands occur.

The moment generating function of $\bfomega$ is
\begin{equation*}\Expect(e^{\langle
    \bfomega,\bfs\rangle})=[z^{n}]A(e^{s_{1}},e^{s_{2}},z).\end{equation*}

For extracting the coefficient, we investigate the dominant singularity of
$A(x,y,z)$. Since the final component is strongly
connected and aperiodic, we have a unique dominant simple eigenvalue of
$\sum_{\eps\in\mathcal A_{I}}x^{\eps}M_{\eps}(y)$ at $(x,y)^{t}=\bfone$ by the
theorem of Perron-Frobenius
(cf.~\cite{Godsil-Royle:2001:alggraphtheory}). Because the final component is complete, this dominant
eigenvalue is $K$, that is the size of the input alphabet $\mathcal A_{I}$. Thus, the unique dominant singularity of $f(x,y,z)^{-1}=\det\big(\mmatrix\big)^{-1}$
at $(x,y)^{t}=\bfone$
is a simple pole at $\rho(\bfone)=1$. Therefore, we have
$f_{z}(\bfone)\neq0$.

For $(x,y)^{t}$ in a small neighborhood of $\bfone$, there is a unique
dominant singularity $\rho(x,y)$ of $f(x,y,z)^{-1}$ due to the continuity of
eigenvalues.

Next, we consider the non-final components of the transducer. The
corresponding transducer $\T_{0}$ is not complete. Let $\T_{0}^{+}$ be the
complete transducer that is obtained from $\T_{0}$ by adding loops where necessary. The
dominant eigenvalue of $\T_{0}^{+}$ is $K$. As the
corresponding sums of transition matrices of $\T_{0}$ and $\T_{0}^{+}$ satisfy element-wise inequalities
but are not equal (at $(x,y)^{t}=\bfone$), the theorem of Perron-Frobenius
(cf.~\cite[Theorem 8.8.1]{Godsil-Royle:2001:alggraphtheory}) implies that the
dominant eigenvalues of $\T_{0}$ have absolute value less than $K$.
Thus, the dominant singularities of $F_{2}(1,1,z)^{-1}$ are at $\lvert
z\rvert>1$. By continuity, this also holds for a small neighborhood of $(x,y)^{t}=\bfone$.

As $A(1,1,z)=(1-z)^{-1}$, we obtain $F_{1}(\bfone)\neq 0$ and
$F_{1}(x,y,\rho(x,y))\neq 0$ for $(x,y)^{t}$ in a small neighborhood of
$\bfone$. Therefore, $\rho(x,y)$ is the simple dominant pole of $A(x,y,z)$ in
a small
neighborhood of $\bfone$.

The Laurent series of $A(x,y,z)$ at $z=\rho(x,y)$ is
\begin{equation*}A(x,y,z)=(z-\rho(x,y))^{-1}C(x,y)+\text{ power series in }(z-\rho(x,y))\end{equation*}
for a function $C(x,y)$ which is analytic in a neighborhood of $\bfone$ with
$C(\bfone)\neq 0$.
Thus, by singularity analysis~\cite{Flajolet-Sedgewick:ta:analy}, we have
\begin{equation*}\Expect(e^{\langle
  \bfomega,\bfs\rangle})=[z^{n}]A(e^{s_{1}},e^{s_{2}}, z)=e^{u(\bfs)n+v(\bfs)}\big(1+\bigOh(\kappa^{n})\big)\end{equation*}
with
\begin{align*}
u(\bfs)&=-\log\rho(e^{\bfs}),\\
v(\bfs)&=\log(-C(e^{\bfs})\rho(e^{\bfs})^{-1})
\end{align*}
and $\kappa<1$.

Theorem~\ref{thm:quasi} yields the expected value, the variance-covariance
matrix and the asymptotic normality of $\bfomega$. By implicit
differentiation, we obtain the stated expressions. The error terms for the
input sum are $0$ because the input letters are independent and identically
distributed. This also yields the explicit constants in \eqref{eq:e-1-v-1-explicit-trivial}.

Since the input alphabet $\mathcal A_{I}$ has at least two elements, the input sum
has nonzero asymptotic variance. Thus, the asymptotic variance-covariance matrix $\Sigma$ can have rank $1$ or
$2$. Now, we consider these two cases separately and prove the asserted equivalence.
\begin{enumerate}
\item Let $\Sigma$ have rank $1$. Then $\bfomega$ converges to a
  degenerate and a normally distributed random variable  if the asymptotic
  variance of the output sum is $0$; or a linear
  transformation thereof otherwise. Thus, $\bfomega$ is asymptotically independent if
  and only if the asymptotic variance of the sum of the output is $0$. As the
  rank of $\Sigma$ is $1$, the asymptotic variance is $0$ if and only
  if the asymptotic covariance is $0$.
\item Let $\Sigma$ be invertible. By Theorem~\ref{thm:quasi}, we obtain an asymptotic joint normal
 distribution. Thus, $\bfomega$ is asymptotically independent if and only if
 its asymptotic covariance is $0$.
\end{enumerate}
\end{proof}

\subsection{Combinatorial characterization of independent transducers}
To obtain the combinatorial characterization, we use a version of the Matrix-Tree Theorem as proved by Chaiken~\cite{Chaiken:1982:matrixtree}
and Moon~\cite{Moon:1994:matrixtree}. This version does not use trees, but
\emph{forests}, i.e., digraphs whose weak components are trees.

\begin{definition}Let $A$, $B\subseteq\{1,\ldots, N\}$. Let $\mathcal F_{A,B}$ be
  the set of all forests
  which are spanning subgraphs of the final component of the transducer $\T$ with $|A|$ trees such that
every tree is rooted at some vertex $a\in A$ and contains exactly one
vertex $b\in B$.

Let $A=\{i_{1},\ldots, i_{n}\}$ and $B=\{j_{1},\ldots,
  j_{n}\}$ with $i_{1}<\cdots<i_{n}$ and $j_{1}<\cdots<j_{n}$. For $F\in\mathcal F_{A,B}$, we define a function $g\colon B\to A$ by $g(j)=i$ if $j$ is in the tree of
$F$ which is rooted in vertex $i$. We further define the function
$h\colon A\to B$ by $h(i_{k})=j_{k}$ for $k=1,\ldots,n$. The composition
$g\circ h\colon A\to A$ is a permutation on $A$. We define $\sign F=\sign
g\circ h$.
\end{definition}

If $|A|\neq|B|$, then $\mathcal F_{A,B}=\emptyset$. If $|A|=|B|=1$, then
$\sign F=1$ and $\mathcal F_{A,B}$ consists of all spanning trees
rooted in $a\in A$.

\begin{theorem*}[All-Minors-Matrix-Tree Theorem~\cite{Chaiken:1982:matrixtree,Moon:1994:matrixtree}]
For  a directed
graph with loops, let $L=(l_{ij})_{1\leq i,j\leq N}$ be the Laplacian matrix, that is $\sum_{j=1}^{N}l_{ij}=0$ for every $i=1,\ldots,N$ and
 $-l_{ij}$ is the number of edges from $i$ to $j$ for $i\neq j$.
Then, for $|A|=|B|$, the minor $\det L_{A,B}$ satisfies
\begin{equation*}\det L_{A,B}=(-1)^{\sum_{i\in A}i+\sum_{j\in B}j}\sum_{F\in \mathcal
  F_{A,B}}\sign F\end{equation*}
where $L_{A,B}$ is the matrix $L$ whose rows with index in $A$ and columns with
index in $B$ are deleted.
\end{theorem*}

The All-Minors-Matrix-Tree Theorem is still valid for $|A|\neq |B|$ if we assume
that the determinant of a non-square matrix is $0$. For notational simplicity,
we use this convention in the rest of this section.

The next lemma connects the derivatives of $f(x,y,z)$ with weighted sums
of functional digraphs. Theorem~\ref{thm:comb} follows immediately from this lemma and Theorem~\ref{thm:alg}.

\begin{lemma}\label{lem:deriv}
For $f(x,y,z)=\det\big(\mmatrix\big)$, we have
\begin{gather*}\begin{aligned}
    f_{x}(1,1,1) &=-K^{-N}\eps(\mathcal D_{1}), & \!\qquad f_{xy}(1,1,1)
    &=K^{-N}(\eps\delta(\mathcal D_{2})-\eps\delta(\mathcal D_{1})), \\
    f_{y}(1,1,1) &=-K^{-N}\delta(\mathcal D_{1}), & \!\qquad f_{xz}(1,1,1)
    &=K^{-N}(\eps\mathds1(\mathcal D_{2})-\eps\mathds1(\mathcal D_{1})), \\
    f_{z}(1,1,1) &=-K^{-N}\mathds1(\mathcal D_{1}), & \!\qquad
    f_{yz}(1,1,1) &=K^{-N}(\delta\mathds1(\mathcal D_{2})-\delta\mathds1(\mathcal
    D_{1})),    \end{aligned}\\
    \begin{aligned}
      f_{xx}(1,1,1)+f_{x}(1,1,1)&=K^{-N}(\eps\eps(\mathcal D_{2})-\eps\eps(\mathcal
      D_{1})),\\
      f_{yy}(1,1,1)+f_{y}(1,1,1)&=K^{-N}(\delta\delta(\mathcal D_{2})-\delta\delta(\mathcal
      D_{1})),\\
      f_{zz}(1,1,1)+f_{z}(1,1,1)&=K^{-N}(\mathds1\mathds1(\mathcal D_{2})-\mathds1\mathds1(\mathcal D_{1})).\\
    \end{aligned}\end{gather*}
\end{lemma}
\begin{proof}
 The idea of the proof is as follows: First, we compute the derivatives and write them as
  sums over all states. Using the All-Minors-Matrix-Tree Theorem, we change the summation to a sum over
 forests. In the next step, we again change to a sum over functional digraphs.

 Let $u_{1}$, $u_{2}$ be any of the variables $x$, $y$ or $z$. For a matrix $M=(m_{ij})_{1\leq i,j\leq N}$,
 we define the matrix $M_{k:u_{1}}=(\hat{m}_{ij})_{1\leq i,j\leq N}$ with
 $\hat{m}_{ij}=m_{ij}$ for $i\neq k$ and $\hat{m}_{kj}=\frac{\partial}{\partial
 u_{1}}m_{kj}$. Thus $M_{k:u_{1}}$ is the matrix $M$ where row $k$ is
 differentiated with respect to $u_{1}$. 

 We further define the derivatives at $\bfone$ as 
 \begin{align*}D_{u_{1}}(\,\cdot\,)&=\frac{\partial}{\partial
       u_{1}}(\,\cdot\,)\Big\vert_{\bfone}\\\intertext{and}
 D_{u_{1}u_{2}}(\,\cdot\,)&=\frac{\partial^{2}}{\partial
       u_{1}\partial u_{2}}(\,\cdot\,)\Big\vert_{\bfone}.
\end{align*}

Applying the product rule to the definition
of the determinants gives us
\begin{align*}D_{u_{1}}(f)&=\sum_{j=1}^{N}\det\Big(\mmatrix\Big)_{j:u_{1}}\Big\vert_{\bfone},\\
D_{u_{1}u_{2}}(f)&=\sum_{i=1}^{N}\sum_{j=1}^{N}\det\Big(\mmatrix\Big)_{i:u_{1},\,j:u_{2}}\Big\vert_{\bfone}.\end{align*}
In these equations, we have a sum over all states.

Since our original matrix $\mmatrix$ is sparse, and $(\mmatrix)_{j:u_{1}}$
is even sparser, we use Laplace expansion along row $j$ to determine these
determinants. If $i\neq j$, we use Laplace expansion along row
$i$ and $j$ to determine $\det(\mmatrix)_{i:u_{1},\,j:u_{2}}$ for the second
derivatives. If $i=j$, we only expand along row $j$. Depending on the
variable of differentiation, there are at most $K$ nonzero values in row $j$
after differentiation.

For a transition $e$, we denote by $t(e)$, $h(e)$, $\eps(e)$ and $\delta(e)$
the tail, the head, the input and the output of the
transition $e$, respectively. Furthermore, let $w_{e}=\frac1Kx^{\eps(e)}y^{\delta(e)}z$ be the
weight of the transition $e$.

If we use Laplace expansion along two different rows, we must be careful with the
sign. Therefore, we define 
\begin{equation*}\sigma_{de}=(-1)^{[t(e)>t(d)]+[h(e)>h(d)]}\end{equation*}
for two transitions $d$ and $e$. Here, we use Iverson's notation, that is
$[$\emph{expression}$]$ is $1$ if \emph{expression} is true and $0$ otherwise (cf.~\cite{Graham-Knuth-Patashnik:1994}).

Let $L$ be the Laplacian matrix of the underlying graph, that is
\begin{equation*}L=KI-\sum_{\eps\in\mathcal A_{I}}M_{\eps}(1).\end{equation*}

Recall the notation $L_{A,B}$ for the matrix where the rows corresponding to $A$ and the columns corresponding to $B$ have been removed. Laplace expansion yields
 \begin{align*}
 D_{u_{1}}(f)&=-K^{-N+1}\sum_{j=1}^{N}\sum_{\substack{e\in\mathcal E\\ t(e)=j}}(-1)^{t(e)+h(e)}D_{u_{1}}(w_{e})\det(L_{\{t(e)\},\{h(e)\}}),\\
 D_{u_{1}u_{2}}(f)&=-K^{-N+1}\sum_{j=1}^{N}\sum_{\substack{e\in\mathcal E\\
   t(e)=j}}(-1)^{t(e)+h(e)}D_{u_{1}u_{2}}(w_{e})\det(L_{\{t(e)\},\{h(e)\}})\\
 &\qquad+K^{-N+2}\sum_{i=1}^{N}\sum_{\substack{j=1\\j\neq
     i}}^{N}\sum_{\substack{d\in\mathcal E\\ t(d)=i}}\sum_{\substack{e\in\mathcal E\\
   t(e)=j}}\big((-1)^{t(d)+h(d)+t(e)+h(e)}\sigma_{de}\\&\hspace{8em}\cdot D_{u_{1}}(w_{d})D_{u_{2}}(w_{e})\det(L_{\{t(d),t(e)\},\{h(d),h(e)\}})\big).
 \end{align*}

Next, we use the All-Minors-Matrix-Tree Theorem and change the summation over all rows to a summation over forests. We obtain
 \begin{align*}
 D_{u_{1}}(f)&=-K^{-N+1}\sum_{e\in\mathcal
   E}D_{u_{1}}(w_{e})\sum_{F\in\mathcal F_{\{t(e)\},\{h(e)\}}}\sign F,\\
 D_{u_{1}u_{2}}(f)&=-K^{-N+1}\sum_{e\in\mathcal
   E}D_{u_{1}u_{2}}(w_{e})\sum_{F\in\mathcal F_{\{t(e)\},\{h(e)\}}}\sign F\\
 &\qquad+K^{-N+2}\sum_{d\in\mathcal E}\sum_{\substack{e\in\mathcal E\\e\neq d}}\Big(\sigma_{de}D_{u_{1}}(w_{d})D_{u_{2}}(w_{e})\\&\hspace{8em}\sum_{F\in\mathcal
 F_{\{t(d),t(e)\},\{h(d),h(e)\}}}\sign F\Big).
 \end{align*}

Let $F\in\mathcal F_{\{t(e)\},\{h(e)\}}$ be a forest for a transition $e\in\mathcal E$. Then
$F+e$ is a spanning functional digraph with one component. Let $F\in\mathcal
F_{\{t(d),t(e)\},\{h(d),h(e)\}}$ be a forest for transitions $d,e\in\mathcal E$. Then
$F+d+e$ is a spanning functional digraph with
one or two components, depending on $\sigma_{de}\sign
F$. If $\sigma_{de}\sign
F=1$, then it has two components. Otherwise, it has one component. Now we
can change the summation into a sum over functional digraphs and obtain 
 \begin{align*}
 D_{u_{1}}(f)&=-K^{-N+1}\sum_{D\in\mathcal D_{1}}\sum_{C\in\mathcal
   C_{D}}\sum_{e\in C}D_{u_{1}}(w_{e}),\\
 D_{u_{1}u_{2}}(f)&=-K^{-N+1}\sum_{D\in\mathcal
   D_{1}}\sum_{C\in\mathcal C_{D}}\sum_{e\in C}D_{u_{1}u_{2}}(w_{e})\\
 &\qquad+K^{-N+2}\sum_{D\in\mathcal
   D_{2}}\sum_{C_{1}\in\mathcal C_{D}}\sum_{\substack{C_{2}\in\mathcal
     C_{D}\\C_{2}\neq C_{1}}}\sum_{d\in
   C_{1}}\sum_{e\in C_{2}}D_{u_{1}}(w_{d})D_{u_{2}}(w_{e})\\
&\qquad-K^{-N+2}\sum_{D\in\mathcal
   D_{1}}\sum_{C\in\mathcal C_{D}}\sum_{d\in C}\sum_{\substack{e\in C\\e\neq d}}D_{u_{1}}(w_{d})D_{u_{2}}(w_{e}).
 \end{align*}

For a transition $e$, we know the first derivatives
\begin{alignat*}{6}
D_{x}(w_{e})&=\frac1K\eps(e),&\qquad
D_{y}(w_{e})&=\frac1K\delta(e),&\qquad
D_{z}(w_{e})&=\frac1K\mathds 1(e),
\end{alignat*}
and the second derivatives
\begin{alignat*}{4}
D_{xy}(w_{e})&=\frac1K\eps(e)\delta(e),&\qquad
D_{xx}(w_{e})&=\frac1K\eps(e)(\eps(e)-1),\\
D_{xz}(w_{e})&=\frac1K\eps(e)\mathds 1(e),&\qquad
D_{yy}(w_{e})&=\frac1K\delta(e)(\delta(e)-1),\\
D_{yz}(w_{e})&=\frac1K\delta(e)\mathds 1(e),&\qquad D_{zz}(w_{e})&=0.
\end{alignat*}

Thus, we obtain the formulas stated in the lemma.
\end{proof}

\subsection{Bounded Variance and singular asymptotic vari\-ance-\-co\-vari\-ance
  matrix}
We next give the proof of the equivalence of the three statements in Theorem~\ref{thm:var}, including
the bounded variance.

\begin{proof}[Proof of Theorem~\ref{thm:var}]
We first prove~\itemref{it:var0} $\Leftrightarrow$~\itemref{it:somewalks} by giving an alternative representation of the generating
function $A(x,y,z)$ from the proof of Theorem~\ref{thm:alg}. Then we prove the
equivalence \itemref{it:somewalks}~$\Leftrightarrow$~\itemref{it:allcycles}.

\begin{description}[font=\normalfont]
\item[\itemref{it:var0} $\Leftrightarrow$ \itemref{it:somewalks}]
WLOG, we assume that the expected value $\Expect
(\out(X_{n}))$ is a $\bigOh(1)$. Otherwise, we have
$\Expect(\out(X_{n}))=e_{2}n+\bigOh(1)$  for some constant $e_{2}$ (see Theorem~\ref{thm:alg}). Then we subtract $e_{2}$ from the output
of every transition, as for Corollary~\ref{cor:comb-indep}. Under this assumption,
Theorem~\ref{thm:comb} implies that \itemref{it:somewalks} can only hold with $k=0$.

As the input sum is inconsequential, we consider
$A(1,y,z)$. For brevity, we write $A(y,z)$ instead. We obtain
\begin{equation*}A(y,z)=\boldsymbol{u}^{t}\Big(I-\frac zK\sum_{\eps\in\mathcal A_{I}}M_{\eps}'(y)\Big)^{-1}\boldsymbol{v}\end{equation*}
where $M_{\eps}'$ for $\eps\in\{0,\ldots,q-1\}$ are the transition matrices of $\T$.

Since $\T$ is complete, finally connected and finally aperiodic, $A(1,z)$ has
a simple dominant pole at $z=1$ (see the proof of Theorem~\ref{thm:alg}). We
know that
\begin{equation}\label{eq:exp-and-var}\begin{aligned}
\Expect(\out(X_{n}))&=[z^{n}]A_{y}(1,z)=\bigOh(1),\\
\Var(\out(X_{n}))&=[z^{n}]A_{yy}(1,z)+\bigOh(1).
\end{aligned}\end{equation}

Let $s$ be any state of the final component. Each path starting at state $1$
either does or does not visit state $s$. In
the first case, this path can be decomposed into a path leading to state $s$
and visiting $s$ only once, followed by a sequence of closed walks visiting
state $s$ exactly once, and a path starting in $s$ and not returning to
$s$. We translate this decomposition into an equation for the corresponding
generating functions.

Let $\mathcal P^{s}$ be the set of all walks in $\T$ which start at state $s$
but never return to state $s$. All other states can be visited arbitrarily
often. We define the corresponding generating function $P^{s}(y,z)=\sum_{P\in\mathcal P^{s}}y^{\delta(P)} z^{\mathds
  1(P)}K^{-\mathds 1(P)}$.
Then $[z^{n}]P^{s}(y,z)$ is the probability
generating function of the output sum over walks in $\mathcal P^{s}$ of length $n$. 

Let $\mathcal P^{1s}$ be the set of all walks in $\T$ which start at state $1$
and lead to state $s$, visiting $s$ exactly once. If $s=1$, this set consists only
of the path of length $0$. The corresponding generating function is called $P^{1s}(y,z)$.

Let $\mathcal P^{1}$ be the set of all walks in $\T$ which start at state $1$
and never visit state $s$. If $s=1$, this set is empty. The corresponding
generating function is called $P^{1}(y,z)$.

Let $\mathcal C^{s}$ be the set of all closed walks in $\T$ which visit
state $s$ exactly once. All other states can be visited arbitrarily often. The corresponding
generating function is called $C^{s}(y,z)$.

Thus, we have
\begin{equation}\label{eq:alt-representation-A}A(y,z)=P^{1}(y,z)+\frac{P^{1s}(y,z)P^{s}(y,z)}{1-C^{s}(y,z)}.\end{equation}

Let $\alpha$ be any of the superscripts $1$, $1s$ or $s$. By deleting the
transitions leading to $s$, we have
\begin{equation*}P^{\alpha}(y,z)=(\boldsymbol{u}^{\alpha})^{t}\Big(I-\frac zK\sum_{\eps\in\mathcal A_{I}}M_{\eps}'(y)E\Big)^{-1}\boldsymbol{v}^{\alpha},\end{equation*}
where $E=\diag(1,\ldots,1,0,1,\ldots,1)$ and $\boldsymbol{u}^{\alpha}$ and $\boldsymbol{v}^{\alpha}$ are fixed vectors. The position of the zero on the
diagonal of $E$ corresponds to the state $s$. The vectors $\boldsymbol{u}^{\alpha}$ and $\boldsymbol{v}^{\alpha}$ depend on $\alpha$ and
may include the output of the transitions leading to $s$, but $E$ is
independent of $\alpha$. Since we have the element-wise
inequalities 
\begin{equation*}0\leq \sum_{\eps\in\mathcal A_{I}}M_{\eps}'(1)E\leq
\sum_{\eps\in\mathcal A_{I}}M_{\eps}'(1)\end{equation*}
and
$\sum_{\eps\in\mathcal A_{I}}M'_{\eps}(1)E\neq
\sum_{\eps\in\mathcal A_{I}}M'_{\eps}(1)$, we know that the
spectral radii satisfy
\begin{equation*}\rho\Big(\sum_{\eps\in\mathcal A_{I}}M_{\eps}'(1)E\Big)<\rho\Big(\sum_{\eps\in\mathcal A_{I}}M_{\eps}'(1)\Big)=K\end{equation*}
due to the theorem of
Perron-Frobenius (cf.~\cite[Theorem 8.8.1]{Godsil-Royle:2001:alggraphtheory}). Here, it is important that $s$ lies in the final component. Thus, the
dominant singularities of $P^{\alpha}(1,z)$ are at $\lvert z\rvert>1$. Furthermore,
we know that $P^{s}(1,1)>0$ and $P^{1s}(1,1)>0$ by the definition as generating functions.

Because $z=1$ is a simple pole of $A(1,z)$, no pole of $P^{1}(1,z)$ and $P^{1s}(1,z)P^{s}(1,z)$,  and
$P^{1s}(1,1)P^{s}(1,1)\neq0$, it is a simple root of
$1-C^{s}(1,z)$ by \eqref{eq:alt-representation-A}. Thus, we can write $1-C^{s}(1,z)=(z-1)g(z)$ for a suitable
function $g(z)$ with $g(1)\neq0$.

By~\eqref{eq:exp-and-var},~\eqref{eq:alt-representation-A} and singularity analysis~\cite{Flajolet-Sedgewick:ta:analy}, we obtain
\begin{align*}
\bigOh(1)=\Expect(\out(X_{n}))&=P^{1s}(1,1)P^{s}(1,1)C^{s}_{y}(1,1)g(1)^{-2}n+\bigOh(1).\nonumber
\end{align*}
Therefore, $C^{s}_{y}(1,1)=0$.

Similarly, we have
\begin{align}
\Var(\out(X_{n}))&=P^{1s}(1,1)P^{s}(1,1)C^{s}_{yy}(1,1)g(1)^{-2}n+\bigOh(1),\label{eq:var0}
\end{align}
taking into account that $C^{s}_{y}(1,1)=0$.

By~\eqref{eq:var0}, $\Var(\out(X_{n}))=\bigOh(1)$ is equivalent to $C^{s}_{yy}(1,1)=0$, and thus, $C^{s}_{yy}(1,1)+C^{s}_y(1,1)=0$
as $C^{s}_y(1,1)=0$. By the definition of $C^{s}(y,z)$, this is equivalent to
\begin{equation*}
\sum_{C\in\mathcal C^{s}}\delta(C)^{2}K^{-\mathds1(C)}=0,
\end{equation*}
and thus $\delta(C)=0$ for all $C\in\mathcal C^{s}$.
\item[\itemref{it:somewalks} $\Rightarrow$ \itemref{it:allcycles}]
Let $\mathcal C^{s}$ be the set of all closed walks in the final component of $\T$ which visit
state $s$ exactly once. If $D$ is any cycle of the final component of the transducer, then one of the following occurs.
\begin{itemize}
\item No visits of state $s$: Let $i$ be a vertex of $D$. Because the final component is strongly
  connected, there exists a closed walk $C\in\mathcal C^{s}$ with $s$, $i\in C$. Let
  $D'$ be the combined closed walk of $D$ and $C$. Then, $D'\in\mathcal C^{s}$, and so we have 
  \begin{equation*}\delta(D)=\delta(D')-\delta(C)=k\mathds 1(D')-k\mathds 1(C)=k\mathds 1(D).\end{equation*}
\item One visit of state $s$: Then we have $D\in\mathcal C^{s}$ and $\delta(D)=k\mathds 1(D)$.
\end{itemize}
\item[\itemref{it:allcycles} $\Rightarrow$ \itemref{it:somewalks}]As a closed walk
  visiting $s$ exactly once can be decomposed into cycles, this is obvious.
\end{description}
\end{proof}

Next, we prove the equivalence for the quasi-deterministic output sum.

\begin{proof}[Proof of Theorem~\ref{thm:quasi-det}]\ 

\begin{description}[font=\normalfont]
\item[\itemref{it:quasidet} $\Rightarrow$ \itemref{it:allcycles-nonfinal}]Let $C$ be
  an arbitrary cycle of the transducer and $P$ be a path from the initial
  state $1$ to any state of the cycle. Let $z_{n}$ be the input sequence along
  the combined walk consisting of $P$ and $n$ times $C$. Then, by quasi-determinism and the definition of the output, we have
\begin{equation*}k(\mathds 1(P)+n\mathds 1(C))+\bigOh(1)=\out(z_{n})=\delta(P)+n\delta(C)+\bigOh(1).
\end{equation*}
Thus, $n(\delta(C)-k\mathds 1(C))$ is bounded by a constant depending on $P$ and
$C$, but independent of $n$. Therefore, we know that $\delta(C)=k\mathds 1(C)$.
\item[\itemref{it:allcycles-nonfinal} $\Rightarrow$ \itemref{it:quasidet}]WLOG, we assume $k=0$ (replace $\delta(e)$ by $\delta(e)-k$ for all transitions $e$).
For every $z\in\mathcal A_{I}^{*}$, we have
$\lvert\out(z)\rvert\leq\sum_{e\in\mathcal
  E}\lvert\delta(e)\rvert+\max_{s\in\{1,\ldots, S\}}\lvert a(s)\rvert$ because
all cycles have output sum $0$ so that every transition contributes at most
once to $\out(z)$. Therefore, we have a quasi-deterministic random variable
$\out(X_{n})=\bigOh(1)$.
\end{description}
\end{proof}

Now, we consider transducers whose output alphabet is $\{0,1\}$ and prove that
there are only trivial cases with a bounded variance.
\begin{proof}[Proof of Corollary~\ref{cor:01output}]
We know that the output digits $(0,\ldots,0)$ and $(1,\ldots,1)$ have
asymptotic variance~$0$.

Assume that the asymptotic variance is $0$. Let $k$ be the constant given in
Theorem~\ref{thm:var}. Then, we know $k\in[0,1]$. By the
aperiodicity of the final component, there exist cycles
$C_{1},\ldots,C_{n}$ of coprime length and therefore integers $b_{1},\ldots,b_{n}$ with
\begin{equation*}1=b_{1}\mathds 1(C_{1})+\cdots+b_{n}\mathds1(C_{n}).\end{equation*}
Thus,
\begin{equation*}k=b_{1}\delta(C_{1})+\cdots+b_{n}\delta(C_{n})\in\Z\end{equation*}
and hence, $k\in\{0,1\}$. Therefore, $(0,\ldots,0)$ and $(1,\ldots,1)$ are
the only output digits with asymptotic variance $0$.
\end{proof}

This last proof shows the equivalence of the statements in
Corollary~\ref{thm:reg}, including a transducer with a singular asymptotic
variance-covariance matrix.
\begin{proof}[Proof of Corollary~\ref{thm:reg}]WLOG, we assume that both
  expected values $\Expect(\out(X_{n}))$ and $\Expect(\inputsum(X_{n}))$ are $\bigOh(1)$.

We know
that the asymptotic variance $v_{1}$ of the input is non-zero because $\mathcal A_{I}$
consists of at least two elements. As in the last paragraph of the proof of
Theorem~\ref{thm:quasi}, we consider the random variables
$Y_n = \inputsum(X_{n})$ and $Z_n=-\frac{c}{v_{1}}\inputsum(X_{n})+\out(X_{n})$ and their
variance-covariance matrix $\bigl(\begin{smallmatrix}v_{1}&0\\0&v_{2}-\frac{c^{2}}{v_{1}}\end{smallmatrix}\bigr)$.
The matrix $\Sigma$ is singular if and only if the asymptotic
variance of $Z_{n}$ is $0$.

Thus, we consider a transducer with the same input as the original transducer $\T$
for which the output of a transition $e$ is $-\frac{c}{v_{1}}\eps(e)+\delta(e)$. By Theorem~\ref{thm:var}, the output sum of this new transducer
has asymptotic variance $0$ if and only if there exists an $m\in\R$ such that
\begin{equation*}-\frac{c}{v_{1}}\eps(C)+\delta(C)=m\mathds1(C)\end{equation*}
for every cycle $C$ of the final component. Since the expected value of $Z_{n}$ is $\bigOh(1)$, we have
$m=0$.

The second statement follows from Theorem~\ref{thm:var}.
\end{proof}
\bibliographystyle{amsplain}
\bibliography{cheub}

\end{document}